\documentclass[12pt]{amsart}
\usepackage{latexsym}
\usepackage{amscd, amsfonts, eucal, mathrsfs, amsmath, amssymb, amsthm}
\input xy
\xyoption{all}

\usepackage{pdfsync}
\usepackage{baskervald}
\usepackage{hyperref}
\usepackage[active]{srcltx}

\usepackage[T2A,T1]{fontenc}
\DeclareSymbolFont{cyrillic}{T2A}{cmr}{m}{n}
\DeclareMathSymbol{\Sha}{\mathalpha}{cyrillic}{216}

\newcommand{\field}[1]{\mathbf #1}

\newcommand{\mc}[1]{\mathcal #1}
\newcommand{\ms}[1]{\mathscr #1}
\newcommand{\widebar}[1]{\overline{#1}}

\newcommand{\bbig}{\text{\rm big}}

\newcommand{\R}{\field R}

\renewcommand{\C}{\field C}

\newcommand{\Z}{\field Z}
\newcommand{\Q}{\field Q}

\newcommand{\simto}{\stackrel{\sim}{\to}}

\newcommand{\send}{\ms E\!nd}

\DeclareMathOperator{\End}{End}

\DeclareMathOperator{\tw}{tw}

\DeclareMathOperator{\Spec}{Spec}
\DeclareMathOperator{\spec}{Spec}

\DeclareMathOperator{\Supp}{Supp}

\renewcommand{\phi}{\varphi}

\newcommand{\Az}{\underline{\operatorname{Az}}}

\DeclareMathOperator{\Def}{Def}
\newcommand{\Td}{\operatorname{Td}}

\renewcommand{\P}{\field P}

\newcommand{\A}{\field A}

\DeclareMathOperator{\Pic}{Pic}
\DeclareMathOperator{\Sh}{Sh}

\DeclareMathOperator{\D}{D}

\DeclareMathOperator{\pr}{pr}

\DeclareMathOperator{\Ext}{Ext}

\DeclareMathOperator{\NS}{NS}

\DeclareMathOperator{\PGL}{PGL}

\newcommand{\m}{\boldsymbol{\mu}}

\newcommand{\G}{\field G} 

\renewcommand{\H}{\operatorname{H}}

\DeclareMathOperator{\chern}{ch}

\renewcommand{\]}{]\!\hspace{0.03em}]}

\renewcommand{\)}{)\!\hspace{0.03em})}

\DeclareMathOperator*{\tensor}{\otimes}

\DeclareMathOperator*{\ltensor}{\stackrel{\field L}{\otimes}}
\DeclareMathOperator{\Tr}{\operatorname{Tr}}

\DeclareMathOperator{\rk}{\operatorname{rk}}

\newcommand{\surj}{\twoheadrightarrow}

\newcommand{\inj}{\hookrightarrow}

\DeclareMathOperator{\Aut}{\operatorname{Aut}}

\DeclareMathOperator{\Isom}{\operatorname{Isom}}
\DeclareMathOperator{\Hom}{\operatorname{Hom}}

\DeclareMathOperator{\Br}{\operatorname{Br}}

\DeclareMathOperator{\B}{\operatorname{B\!}}

\DeclareMathOperator{\Jac}{Jac}

\newcommand{\eps}{\varepsilon}

\newcommand{\Chow}{\operatorname{CH}}

\newtheorem{lem}{Lemma}[section]

\newtheorem{thm}[lem]{Theorem}

\newtheorem{prop}[lem]{Proposition}

\newtheorem{cor}[lem]{Corollary}

\newtheorem{claim}[lem]{Claim}

\theoremstyle{definition}
\newtheorem{defn}[lem]{Definition}

\newtheorem{example}[lem]{Example}

\theoremstyle{remark}
\newtheorem{remark}[lem]{Remark}

\numberwithin{equation}{subsection}

\author{Max Lieblich}

\title{Rational curves in the moduli of supersingular K3 surfaces}

\begin{document}
\maketitle
\begin{abstract}
We show how to construct non-isotrivial families of supersingular K3 surfaces over rational curves using a relative form the Artin-Tate isomorphism and twisted analogues of Bridgeland's results on moduli spaces of stable sheaves on elliptic surfaces. As a consequence, we show that every point of Artin invariant 10 in the Ogus space of marked supersingular K3 surfaces lies on infinitely many pairwise distinct rational curves canonically associated to elliptic structures on the underlying K3 surface.
\end{abstract}
\tableofcontents

\section{Introduction}
\label{sec:intro}

In this paper we study certain special rational curves in moduli spaces of K3 surfaces that are generated by cohomology classes. In particular, we focus on two results. Fix an algebraically closed field $k$ of characteristic $p>0$ throughout this paper. Let (for now) $\pi: X\to\spec k$ be the structure map of a supersingular K3 surface.

\begin{thm}[Artin]\label{thmA}
The fppf sheaf $\R^2\pi_\ast\m_p$ is representable by a smooth group scheme over $k$, 
and the connected component $\ms C^\circ$ of the identity is isomorphic to $\G_a$.
\end{thm}

This is a special case of a result of Artin, published as Theorem 3.1 of \cite{MR0371899} with the caveat that the proof would be published elsewhere. I have not found the proof in the literature, so I provide a moduli-theoretic proof here. Using these ``rational curves in cohomology,'' we then prove the following, which is essentially a relative version of the Artin-Tate isomorphism (Theorem 3.1 of \cite{MR1610977}, generalized in Proposition 4.5 of \cite{MR0244271}).

\begin{thm}\label{thmB}
Given a $\m_p$-gerbe $$\ms X\to X\times\A^1$$ inducing an isomorphism $\A^1\simto\ms C^\circ$ (in Theorem \ref{thmA}) and a choice of elliptic fibration $X\to\P^1$, there is a canonically defined open substack of the stack of coherent $\ms X$-twisted sheaves that is a $\G_m$-gerbe over a non-isotrivial family $Y\to\A^1$ of K3 surfaces. Moreover, the structure map of $Y$ admits a factorization  
$$Y\to\P^1\times\A^1$$ such that for each geometric point $t\to\A^1$, the morphism $$Y_t\to\P^1_t$$ is an \'etale form of $$X\tensor\kappa(t)\to\P^1\tensor\kappa(t),$$
and the fiber over $0\in\A^1$ is isomorphic to the original elliptic structure $X\to\P^1$.
Finally, distinct elliptic structures on $X$ give rise to distinct families of K3 surfaces $Y$.
\end{thm}

The family $Y\to\P^1$ in Theorem \ref{thmB} can be made to belong to various moduli problems (the Ogus space of marked K3 surfaces, the space of polarized K3 surfaces, etc.), at least over open subsets of the base $\A^1$. As a consequence, the formation of moduli spaces of twisted sheaves can be used to trace out rational curves on various moduli spaces using rational curves of cohomology classes. 

%
%
%

\subsection{Outline}
In Sections \ref{sec:few-remarks-cohom} and \ref{sec:singular-fibers} 
we prove a few preliminaries about the fppf cohomology of $\m_p$ on families of curves and the singular fibers of elliptic K3 surfaces. 
This is followed by a proof of Theorem \ref{thmA} in Section \ref{sec:rep} and a proof of Theorem \ref{thmB} in Sections \ref{sec:mod-interp} and \ref{sec:ratcurves}. Finally, in Section \ref{sec:static_pencils_and_deformations}, we show that distinct elliptic pencils give rise to distinct rational curves in moduli.

We have included the bare minimum in this manuscript necessary to get the theory off the ground and provide an adequate reference for \cite{1403.3073}. 

\subsection{History}
The work described here has been developing since 2011. I started giving public lectures about it in 2012 (see, for example, the Banff video 
\begin{center}
\url{http://videos.birs.ca/2012/12w5027/201203271601-Lieblich.mp4}\end{center} available since May of 2012) and discussed it by email and in person with various parties in 2012 and 2013. It was also described in my contribution to the 2012 Simons Symposium \cite{MR3114930} (published in April of 2013). 

\subsection{Acknowledgments}
I had many helpful interactions with Valery Alexeev, Brendan Hassett, Daniel Huybrechts, Aise Johan de Jong, Nick Katz, Dino Lorenzini, and Davesh Maulik during the preparation of this manuscript.

\section{A few remarks on the cohomology of
    \texorpdfstring{$\boldsymbol\mu_p$}{} on a family of curves}
\label{sec:few-remarks-cohom}

The following are very well known for \'etale cohomology with coefficients of order prime to
the residue characteristic. They are also true for fppf cohomology with coefficients in $\m_p$, as we
record here.

\begin{lem}
  \label{sec:few-remarks-cohom-2}
  Suppose $C$ is a proper smooth curve over an algebraically closed
  field $k$. The Kummer sequence induces a canonical isomorphism
  $$\Pic(C)/p\Pic(C)=\Z/p\Z\simto\H^2(C,\m_p)$$
\end{lem}
\begin{proof}
  By Tsen's theorem $\H^2(C,\G_m)=0$. Since the $p$th power Kummer
  sequence is exact on the fppf site, the result follows.
\end{proof}

\begin{prop}
  \label{P:curve-coho}
  Suppose $Z\to G$ is a proper smooth morphism of finite presentation of relative dimension $1$ 
  with $G$ connected and $\alpha\in\H^2(Z,\m_p)$. There exists a unique
  element $a\in\Z/p\Z$ such that for every
  geometric point $g\to G$, the restriction $\alpha_{Z_g}\in\Z/p\Z$ is
  equal to $a$ via the isomorphism of Lemma \ref{sec:few-remarks-cohom-2}.
\end{prop}
\begin{proof}
  It suffices to prove this under the assumption that $G$ is the
  spectrum of a complete dvr.
  \begin{lem}
    When $G$ is the spectrum of a complete dvr, we have that $\H^2(Z,\G_m)=0$.
  \end{lem}
  \begin{proof}
    First, since $Z$ is regular the group is torsion.  Thus, any
    $\G_m$-gerbe is induced by a $\m_n$-gerbe for some $n$. Fix a
    $\m_n$-gerbe $\ms Z\to Z$. By Tsen's theorem and Lemma 3.1.1.8 of \cite{MR2388554}, there is
    an invertible $\ms Z_g$-twisted sheaf $L$, where $g$ is the closed
    point of $g$. The obstruction to deforming such a sheaf lies in
    $$\H^2(Z,\ms O)=0,$$
    so that $L$ has a formal deformation over the completion of $\ms
    Z$. By the Grothendieck Existence Theorem for proper Artin stacks,
    Theorem 11.1 of \cite{MR2312554} (or, in this case, the classical Grothendieck Existence Theorem for
    coherent modules over an Azumaya algebra representing $\alpha$),
    this formal deformation algebraizes, trivializing the class of
    $\ms Z$ in $\H^2(Z,\G_m)$, as desired.
  \end{proof}
  Applying the lemma and the Kummer sequence, we see that
  $$\Pic(Z)/p\Pic(Z)\simto\H^2(Z,\m_p).$$
  But restricting to a fiber defines a canonical isomorphism
  $$\Pic(Z)/p\Pic(Z)\simto\Pic(Z_g)/p\Pic(Z_g)\simto\Z/p\Z$$
  independent of the point $g$. The result follows.
\end{proof}

\begin{cor}
  \label{C:coho-class-indep}
  Suppose $E\subset Z\to T$ is a family of smooth genus $1$ fibers in
  a proper flat family of elliptic surfaces of finite presentation
  over a connected base. Given a class
  $$\alpha\in\H^2(Z,\m_p),$$
  there is an element $a\in\Z/p\Z$ such that for every geometric point
  $t\in T$, the restriction of $\alpha$ to $E_t$ equals $a$ via the
  isomorphism of Lemma \ref{sec:few-remarks-cohom-2}.
\end{cor}

\begin{example}
  This kind of thing is not as utterly trivial as it seems. Consider
  the cuspidal cubic $C$. The Kummer sequence shows
  that there is a class $\alpha\in\H^2(C\times\A^1,\m_p)$ whose value
  over $0$ is trivial and whose value over any other geometric point
  is non-trivial. In fact, there is an isomorphism of fppf functors
  $$\R^2f_\ast\m_p\simto\G_a$$
  where $f:C\to\spec k$ denotes the structure map. When
  multiplication by $p$ on the Picard scheme is ramified (or trivial,
  as in the case of $\G_a$), interesting behavior is possible. This
  makes the study of stable twisted sheaves of rank $1$ on families of
  $\m_p$-gerbes on elliptic surfaces somewhat interesting.
\end{example}

\section{Singular fibers of elliptic K3 surfaces}
\label{sec:singular-fibers}

I am indebted to Aise Johan de Jong for pointing out an error in an earlier version of this section, for telling me about extremal elliptic surfaces, and for suggesting the idea of using $k$-rational $j$-invariants to augment the locus of singular fibers in the extremal case (as used in Corollary \ref{sec:modul-interpr-isom-1} below).

Let $X\to\P^1$ be a proper morphism from a K3 surface with smooth geometrically connected geometric fiber of dimension $1$. Recall that there is an associated Jacobian fibration $\pi:J\to\P^1$ (see, e.g., Section 4 of Chapter 11 of \cite{huybrechts}) that is also a K3 surface, but that possesses a section $\sigma:\P^1\to J$. Using the classification of singular fibers and relative minimality of elliptic fibrations on K3 surfaces, we know that there is an \'etale surjection $U\to\P^1$ such that $X_U$ and $J_U$ are isomorphic (see Corollary 5.5 of Chapter 11 of \cite{huybrechts}). In particular, $X$ and $J$ are fiberwise isomorphic.

We will fix $X\to\P^1$ and $J\to\P^1$ in what follows. Recall the Shioda-Tate formula: if $r_t$ is the number of irreducible components of the fiber of $\pi$ over $t$, then
$$\rho(J) = 2 + \sum_t (r_t-1)+\rk J(k(t)),$$
the last term being the rank of the Mordell-Weil group, the group of rational points on the generic fiber of $\pi$.

\begin{prop}
  \label{sec:sing-fibers-ellipt}
  Either $\pi:J\to\P^1$ has at least $3$ singular fibers or the $j$-invariant of $J_\eta$ does not lie in $k\subset k(t)$.
\end{prop}
\begin{proof}
  Suppose there are $m$ singular fibers $F_1,\ldots,F_m$.  The Kodaira
  classification of singular fibers of minimal elliptic fibrations (Theorem IV.8.2 of \cite{MR1312368})
  shows that every singular fiber $F_i$ with $n_i$ components has
  $\ell$-adic Euler characteristic
$$\chi(F,\Z_\ell)\leq n_i + 1.$$
Since a smooth curve of genus $1$ has Euler characteristic $0$, we
have two inequalities
\begin{align*}
2 - 2m + \sum_{i=1}^m (n_i + 1) &= \rho(X)-\rk J(k(t)) \leq 22 \\
\sum_{i=1}^m\chi(F_i,\Z_\ell)=24 &\leq \sum_{i=1}^m (n_i+1),
\end{align*}
the first coming from the Shioda-Tate formula. Letting $S=\sum (n_i+1)$, this yields
$$24 \leq S\leq 22+2m.$$
We conclude that $m\geq 2$, and if $m=2$ then we must have $\rk J(k(t))=0$. In other words, when $m=2$ the surface $J$ is extremal. By Theorem 6.1(1) of \cite{MR1785328} (which, in spite of the paper's title, does not assume that the base field has characteristic $2$ or $3$), the latter implies that the $j$-invariant of $J_{k(t)}$ does not lie in $k$, for otherwise $J$ would be rational (and we know it is a K3 surface).
\end{proof}

\begin{cor}\label{cor:special-set}
Let $X$ be a supersingular K3 surface and $f:X\to\P^1$ a morphism with smooth connected geometric generic fiber of genus $1$. Then either
\begin{enumerate}
\item $f$ has at least three singular fibers, or
\item for any algebraically closed extension field $k\subset K$ and any element $\alpha\in k$, the set of points $t\in\P^1(K)$ such that $j(X_t)=\alpha$ lies in the image of the extension of scalars map 
$\P^1(k)\inj\P^1(K)$.
\end{enumerate}
\end{cor}
\begin{proof}
As noted in the first paragraph of this section, the singular fibers of $f$ are the same as the singular fibers of the Jacobian fibration $\pi:J\to\P^1$, and for any geometric point $t\to\P^1$ that avoids singular fibers, we have $j(X_t)=j(J_t)$. By Proposition \ref{sec:sing-fibers-ellipt}, if $\pi$ has only two singular fibers then the $j$-invariant map
$$j:\P^1\dashrightarrow\A^1$$
is non-constant, hence quasi-finite. But if $Z\to W$ is any quasi-finite morphism of schemes of finite type over an algebraically closed field $k$ and $k\subset K$ is an algebraically closed extension field, the diagram
$$\xymatrix{Z(k)\ar[r]\ar[d] & W(k)\ar[d]\\
Z(K)\ar[r] & W(K)}$$
is Cartesian. Indeed, the fiber of $Z\to W$ over a $k$-point $p\to W$ is a finite $k$-scheme, whose reduced structure must be isomorphic to a disjoint union of copies of $\spec k$ by the structure theory of finite-dimensional algebras over a field.
\end{proof}

\section{Proof of Artin's result}
\label{sec:rep}

Fix a supersingular K3 surface $X$ over $k$ and let $\pi:X\to\spec k$
denote the structure morphism. In this section, we prove Theorem \ref{thmA} using the stack of Azumaya algebras on $X$.

\begin{thm}\label{T:gpscheme}
  The big fppf sheaf $\R^2\pi_\ast\m_p$ is representable by a smooth group
  scheme over $k$ whose connected component is
  isomorphic to $\G_a$.
\end{thm}


We will write $\ms C=\R^2\pi_\ast\m_p$ in the remainder of this section.

\begin{prop}
  \label{P:diag}
  The diagonal
  $$\ms C\to\ms C\times\ms C$$
  is representable by closed immersions of finite presentation.
\end{prop}
\begin{proof}
  Let $a,b:T\to\ms C$ be two maps corresponding to classes
  $$a,b\in\H^0(T,\R^2\pi_\ast\m_p).$$
  To prove that the locus where $a=b$ is represented by a closed
  subscheme of $T$, it suffices by translation to assume that $b=0$
  and prove that the functor
  $Z(a)$ sending a $T$-scheme $S\to T$ to $\emptyset$ if $a_S\neq 0$
  and $\{\emptyset\}$ otherwise is representable by a closed subscheme
  $Z_a\subset T$.

  First, suppose $a$ is the image of a class $\alpha\in\H^2(X_T,\m_p)$
  corresponding to a $\m_p$-gerbe $\ms X\to X_T$. Since
  $\R^1\pi_\ast\m_p=0$, we see that the functor $Z(a)$ parametrizes
  schemes $S\to T$ such that there is an $\alpha'\in\H^2(S,\m_p)$
  with $$\alpha'|_{X_S}=\alpha|_{X_S}.$$

  Let $\ms P\to T$ be the stack whose objects over $S\to T$ are
  families of $\ms X_S$-twisted
  invertible sheaves $\ms L$ together with isomorphisms $\ms
  L^{\tensor p}\simto\ms O_{\ms X_S}$. The stack $\ms P$ is a $\m_p$-gerbe over a
  quasi-separated algebraic space $P\to T$ that is locally of finite
  presentation (Proposition 2.3.1.1 of \cite{MR2309155} and Section C.23 of \cite{MR2786662}). Moreover, since $\Pic_X$ is torsion free, the
  natural map $P\to T$ is a monomorphism.
  Note that if we change $\alpha$ by the preimage
  of a class $\alpha'\in\H^2(T,\m_p)$ we do not change $P$ 
  (but we do change the class of the gerbe $\ms P\to P$ by $\alpha'$).

  \begin{lem}
    The algebraic space $P\to T$ is a closed immersion of finite presentation.
  \end{lem}
  \begin{proof}
    First, let us show that $P$ is of finite presentation.  It
    suffices to show that $P$ is quasi-compact under the assumption
    that $T$ is affine. Moreover, since $\ms C$ is locally of finite
    presentation, we may assume that $T$ is Noetherian. By Gabber's
    Theorem (Theorem 1.1 of \cite{dejong-gabber}) there is a Brauer-Severi scheme $V\to X$ such that
    $\alpha|_V$ has trivial Brauer class, i.e., so that there is an
    invertible sheaf $L\in\Pic(V)$ satisfying
    $$\alpha|_V=[L]^{1/p}\in\H^2(V,\m_p),$$
    where $[L]^{1/p}$ denotes the $\m_p$-gerbe over $V$ parametrizing
    $p$th roots of $L$. Writing $$\ms V=\ms X\times_{X_T} V,$$ we know
    from the isomorphism $\ms V\cong [L]^{1/p}$ that there is an
    invertible $\ms V$-twisted sheaf $\ms L$ such that
    $$\ms L^{\tensor p}\cong L.$$

    Let $W$ denote the algebraic space parametrizing invertible $\ms
    V$-twisted sheaves whose $p$th tensor powers are trivial. By
    the argument in the preceding paragraph, tensoring with $\ms
    L^{\vee}$ defines an isomorphism between $W$ and the fiber of the
    $p$th power map
    $$\Pic_{V/T}\to\Pic_{V/T}$$
    over $[L^\vee]$. Since the $p$th power map is a closed immersion
    ($X$ being K3), we
    see that $W$ is of finite type.  The following lemma then applies
    to show that $P$ is of finite type.

    \begin{lem}
      The pullback map $\Pic^{\tw}_{\ms X/T}\to\Pic^{\tw}_{\ms V/T}$
      is of finite type.
    \end{lem}
    \begin{proof}
      It suffices to prove the corresponding results for the stacks of
      invertible twisted sheaves. Since $V\times_X V$ and $V\times_X
      V\times_X V$ are proper over $T$ this follows from descent
      theory: the category of invertible $\ms X$-twisted sheaves is
      equivalent to the category of invertibe $\ms V$-twisted sheaves
      with a descent datum on $\ms V\times_{\ms X}\ms V$. Thus, the
      fiber over $L$ on $\ms V$ is a locally closed subspace of
      $$\Hom_{\ms V\times_{\ms X}\ms V}(\pr_1^\ast L,\pr_2^\ast L).$$
      Since the latter is of finite type (in fact, a cone in an affine
      bundle), the result follows.
    \end{proof}

    We claim that $P$ is proper over $T$. To see this, we may use the
    fact that it is of finite presentation (and everything is of
    formation compatible with base change) to reduce to the case in
    which $T$ is Noetherian, and then we need only check the valuative
    criterion over dvrs. Thus, suppose $E$ is a dvr with fraction
    field $F$ and $p:\spec F\to P$ is a point. Replacing $E$ by a finite
    extension, we may assume that $p$ comes from an invertible $\ms
    X_F$-twisted sheaf $\ms L$. Taking a reflexive extension and using
    the fact that $\ms X_E$ is locally factorial, we see that $\ms L$
    extends to an invertible $\ms X_E$-twisted sheaf $\ms L_E$. Since
    $\Pic_X$ is separated, it follows that $\ms L_E$ induces the
    unique point of $P$ over $E$ inducing $p$.

    Since a proper monomorphism is a closed immersion, we are done.
  \end{proof}

We now claim that the locus $Z(a)\subset T$ is represented by the
closed immersion $P\to T$.

\begin{lem}
  A class $\alpha\in\H^2(X_T,\m_p)$ represented by a $\m_p$-gerbe $\ms
  X\to X_T$ is trivial if and only there is an invertible $\ms
  X$-twisted sheaf $\ms L$ such that $$\ms L^{\tensor p}\cong\ms O_{\ms
  X}.$$
\end{lem}
\begin{proof}
  Given a scheme $S$, the gerbe $\B\m_{p, S}\to S$ represents the stack of pairs $(\ms L, \phi)$ with $\ms L$ an invertible sheaf and $\phi:\ms L^{\tensor p}\simto\ms O$ is a trivialization. By definition, $\alpha$ is trivial if and only if there is an isomorphism of stacks $\ms X\simto\B\m_{p, X_{T}}$. The result follows.
\end{proof}

Now consider the $\m_p$-gerbe $\ms P\to P$. Subtracting the pullback
from $\alpha$ yields a class such that the associated Picard stack
$\ms P'\to P$ is trivial, whence there is an invertible twisted sheaf
with trivial $p$th power. In other words,
$$a|_P=0\in\ms C(P).$$
On the other hand, if $a|_S=0$ then up to changing $a$ by the pullback
of a class from $S$, there is an invertible twisted with trivial $p$th
power. But this says precisely that $S$ factors through the moduli
space $P$, as desired.
\end{proof}

Let $\Az$ be the $k$-stack whose objects over $T$ are Azumaya
algebras $\ms A$ of degree $p$ on $X_T$ such that for every geometric
point $t\to T$, we have
$$\ker(\Tr:\H^2(X,\ms A)\to\H^2(X,\ms O))=0.$$
It is well known that $\Az$ is an Artin stack locally of finite type over $k$
(see, for example, Lemma 3.3.1 of \cite{MR2579390}, and note that the trace condition is open). 
There is a morphism of stacks
$$\chi:\Az\to\R^2\pi_\ast\m_p$$
(with the latter viewed as a stack with no non-trivial automorphisms
in fiber categories)
given as follows. Any family $\ms A\in\Az_T$ has a corresponding class
$$[\ms A]\in\H^1(X_T,\PGL_p).$$
The non-Abelian coboundary map yields a class in $\H^2(X_T,\m_p)$
which has a canonical image
$$\chi(\ms A)\in\H^0(T,\R^2\pi_\ast\m_p).$$

\begin{prop}
  \label{P:surj-by-az}
  The morphism $\chi$ described above is representable by smooth Artin stacks.
\end{prop}
\begin{proof}
  Since $\Az$ is locally of finite type over $k$ and the diagonal of
  $\ms C$ is representable by closed immersions of finite
  presentation, it suffices to show that $\chi$ is formally smooth.
  Suppose $A'\to A$ is a square-zero extension of Noetherian rings and
  consider a diaagram of solid arrows
  $$\xymatrix{\Spec A'\ar[r]\ar[d] & \Az\ar[d]\\
    \Spec A\ar[r]\ar@{-->}[ur] & \ms C.}$$
  We wish to show that we can produce the dashed diagonal
  arrow. Define a stack $\ms S$ on $\Spec A$ whose objects over an
  $A$-scheme $U\to\spec A$ are dashed arrows in the restricted diagram
    $$\xymatrix{U'\ar[r]\ar[d] & \Az\ar[d]\\
      U\ar[r]\ar@{-->}[ur] & \ms C,}$$ where $U'=U\tensor_A A'$, and
    whose morphisms are isomorphisms between the objects of $\Az$ over
    $U$ restricting to the identity on the restrictions to $U'$.

    \begin{claim}
      The stack $\ms S\to\Spec A$ is an fppf gerbe with coherent band.
    \end{claim}
    \begin{proof}
      First, it is clear that $\ms S$ is locally of finite
      presentation. Suppose $U$ is the spectrum of a complete local
      Noetherian ring with algebraically closed residue field. Then
      $\H^2(U,\m_p)=0$ and the section $U\to\ms C$ is equivalent to a
      class $$\alpha\in\H^2(X_U,\m_p).$$ Let $\ms X\to X_U$ be a
      $\m_p$-gerbe representing $\alpha$ and write $\ms X'=\ms
      X\tensor_A A'$ the restriction of $\ms X$ to $U'$. An object of
      $\Az_{U'}$ is then identified with $\send(V)$ where $V$ is a
      locally free $\ms X'$-twisted sheaf of rank $p$ with trivial
      determinant. The obstruction to deforming such a sheaf lies in
      $$\ker(\H^2(X_{U'},\send(V)\tensor I)\to\H^2(X_{U'},\ms O\tensor
      I)),$$
      and deformations are a pseudo-torsor under
      $$\ker(\H^1(X_{U'},\send(V)\tensor I)\to\H^1(X_{U'},\ms O\tensor
      I))=\H^1(X_{U'}\send(V)\tensor I).$$
      Standard arguments starting from the assumption on the geometric
      points of $\Az$ show that the obstruction group is trivial,
      while the band is the coherent sheaf $\R^1\pi_\ast\send(V)$, as desired.
    \end{proof}
Since any gerbe with coherent band over an affine scheme is neutral,
we conclude that $\ms S$ has a section. In other words, a dashed arrow
exists, as desired.
\end{proof}

\begin{proof}[Proof of Theorem \ref{T:gpscheme}]
  Using Proposition \ref{P:surj-by-az}, a smooth cover $B\to\Az$ gives
  rise a to a smooth cover $B\to\ms C$. Thus, by Proposition
  \ref{P:diag}, $\ms C$ is a separated algebraic group-space locally
  of finite type. Since $k$ is algebraically closed, it follows that
  $\ms C$ is in fact a group scheme locally of finite type. Finally,
  since $X$ is supersingular we have that the completion of $\ms C$ at
  the identity section is isomorphic to $\widehat{\G_a}$, which is
  formally smooth and $p$-torsion. It follows that $\ms C$ is smooth
  over $k$ with $1$-dimensional $p$-torsion connected component. The
  only $p$-torsion smooth $1$-dimensional $k$-group scheme is $\G_a$, completing the proof.
\end{proof}

\section{Modular interpretation of the isomorphism $\operatorname{Br}\cong\Sha$}
\label{sec:mod-interp}

Fix an elliptic K3 surface $f:X\to\P^1$ over the algebraically closed field
$k$.
Given a $\m_n$-gerbe $\ms X\to
X$, let $$a(\ms X)\in\frac{1}{n}\Z/\Z$$ be the unique element that
corrsponds to the
cohomology class of the restriction of $\ms X$ to any smooth fiber
$E\subset X$ of $f$ under the natural isomorphism
$$\frac{1}{n}\Z/\Z\simto\Z/n\Z.$$ Fix an ample divisor $H\subset X$.

Recall the following special case of a theorem of Artin and Tate (Theorem 3.1 of \cite{MR1610977}).

\begin{thm}[Artin-Tate]\label{sec:modul-interpr-isom}
  The edge map in the $E^2$ term of the Leray spectral sequence for
  $\G_m$ on $f:X\to\P^1$ yields an isomorphism
  $$\Br(X)\to\Sha(k(t),\Pic_{X_{k(t)}/k(t)}),$$
  resulting in a natural surjection
  $$\Sha(k(t),\Jac(X_{k(t)}))\surj\Br(X)$$
  with kernel isomorphic to $\Z/i\Z$, where $i$ is the index of the generic fiber $X_{k(t)}$ over $k(t)$.
\end{thm}
In particular, if $X\to\P^1$ has a section, the latter arrow is an isomorphism.
\begin{proof}
This is Proposition 4.5 (and ``cas particulier (4.6)'') of \cite{MR0244271}.
\end{proof}

It follows by descent theory that any element of $\Sha(k(t),\Jac(X_{k(t)}))$
corresponds to an \'etale form $X'$ of $X$, and $X'$ is also a K3
surface. In this section we will describe this isomorphism
geometrically using the theory of stable twisted sheaves. This
geometric description will allow us to take a varying Brauer class on
$X$ and produce a family of K3 surfaces (that are each forms of a
given elliptic fibration on $X$) in Section \ref{sec:ratcurves}. The
central interest arises from the following corollary.

\begin{cor}\label{sec:modul-interpr-isom-1}
  Let $K/k$ be an algebraically closed extension field, and suppose
  $\alpha\in\Br(X_K)$ is not in the image of the restriction map
  $$\Br(X)\to\Br(X_K).$$
  No \'etale form $X'$ of $X_K$ mapping to $\alpha$ via Theorem
  \ref{sec:modul-interpr-isom} is defined over $k$.
\end{cor}
\begin{proof}
  By functoriality of the Leray spectral sequence, the diagram
  \begin{equation}\label{Eq:jac-sha}
  \xymatrix{\Sha(K(t),\Jac(X_{K(t)}))\ar[r] & \Br(X_K)\\
    \Sha(k(t),\Jac(X_{k(t)}))\ar[u]\ar[r] & \Br(X)\ar[u]}
  \end{equation}
  commutes. 
  
  Since $X'\to\P^1_K$ is a form of $X_K\to\P^1_K$, we know by Corollary \ref{cor:special-set} that 
  either
  \begin{enumerate}
	\item[(1)]  $X'\to\P^1_K$ has at least three singular fibers located over the image of the map 
	  $$\P^1(k)\to\P^1(K),$$ or
	  \item[(2)] for any $\alpha\in k$, the set of points $t\in\P^1(K)$ such that $j(X'_t)=\alpha$ lies in the image of $$\P^1(k)\inj\P^1(K).$$
	  \end{enumerate}
  If $X'$ is defined over $k$ then so is any divisor class, and thus
  the elliptic fibration
  $$X'\to \P^1_K$$
  is a $K$-linear change of basis from an elliptic fibration
  $$X'_0\to\P^1$$
  over $k$. On the other hand, by condition (1) or condition (2) above, the change of basis must send at least three elements of $\P^1(k)\subset\P^1(K)$ into $\P^1(k)$. We conclude that the change of basis transforming $X'_0$ into $X'$ is $k$-linear (as any change of basis is determined by its action on three points, and any injective map on three $k$-points determines a $k$-linear change of basis). In other words, $X'\to\P^1_K$ is isomorphic as an elliptic fibration to the base change of an elliptic fibration over $k$, which we may assume without loss of generality is $X'_0\to\P^1$.
  
    Moreover, if $X'\to\P^1_K$ is a form of $X_K\to\P^1_K$ then
  $X'_0\to\P^1$ must be a form of $X\to \P^1$, since the epimorphism property of the big \'etale $k$-sheaves $$\Isom_{\P^1}(X,X'_0)\to\P^1_k$$ can be detected after the base change via $k\inj K$. It then follows from 
  diagram \eqref{Eq:jac-sha} that $\alpha$ must be in the image of the restriction map,
  which is a contradiction.
\end{proof}

\begin{remark}
The moral of the preceding results: a moving Brauer class gives rise to a moving family of
K3 surfaces (and not merely a moving family of elliptic pencils), at least rationally, i.e., when the base of the family is the spectrum of a field. (Note that when $X$ does not have a section, there is some ambiguity about this family of torsors; if we only work with torsors that are deformations of the trivial torsor this goes away.) We will now show that this is true in a strong sense by showing that the Artin-Tate isomorphism can be made regular over a $k$-scheme, using the moduli of twisted sheaves.
\end{remark}

Let $\Chow(X)$ denote the (graded) Chow group of algebraic cycles on
$X$ up to numerical equivalence. The fundamental class maps define
embeddings of $\Chow(X)$ into $H(X)$, where $H(X)$ is any of the
``usual'' integral cohomology theories (crystalline, $\ell$-adic).

\begin{defn}\label{D:tw ch}
  Given a perfect complex $F$ of $\ms X$-twisted sheaves, the
  \emph{twisted Chern character\/} of $F$ is
  $$\chern(F):=\sqrt[n]{\R\gamma_\ast(F^{\ltensor n})}\in\Chow(X)\tensor\Q.$$
  The \emph{twisted Mukai vector\/} of $F$ is
  $$v(F):=\chern(F)\sqrt{\Td_X}.$$
\end{defn}

It is well known that given a pair of perfect complexes $F$ and $G$ of
$\ms X$-twisted sheaves, the Riemann-Roch theorem holds:
$$\chi(F,G)=\deg(\chern(F^\vee\ltensor G)\cdot\Td_X).$$ Moreover, the
twisted Mukai vector is locally constant in a family of perfect
complexes on $\ms X$ (cf.\ the following discussion and Proposition 2.2.7.22 of \cite{MR2309155}).

Recall the following definition, Definition 2.2.7.6 of \cite{MR2309155}.

\begin{defn}
  Given a perfect complex $F$ of $\ms X$-twisted sheaves, the
  \emph{geometric Hilbert polynomial\/} of $F$ is the function
  $P_F(m)=\deg(\chern(F(m))\cdot\Td_X)$.
\end{defn}
As explained in Section 2.2.7.5 of \cite{MR2309155}, $P_F$ is a numerical polynomial with the usual
properties. In particular, we can use it to define stability and
semistable of sheaves. Write $p_F$ for the \emph{reduced Hilbert
  polynomial\/} given by
$$p_F(m)=\frac{1}{\alpha_d}P_F(m),$$
where $\alpha_d$ is the leading coefficient of $P_F$. (See Definition 2.3.2.3 of \cite{MR2309155}.)

\begin{defn}
  A pure $\ms X$-twisted sheaf $F$ is \emph{stable\/} if for every
  subsheaf $G\subset F$ we have that
  $$p_G(m)<p_F(m)$$
  for all $m$ sufficiently large.
\end{defn}

\begin{example}
  If $F$ is an invertible $\ms X$ sheaf supported on a smooth curve in
  $X$ then $F$ is stable with respect to any polarization.
\end{example}

Just as in the classical case of (untwisted) elliptic surfaces, we
will produce a form of $X$ by taking a moduli space. Recall in what follows that $\ms X$ is a $\m_n$-gerbe 
(to remind us of what $n$ means!).

\begin{defn}
  Let $\ms M_{\ms X}$ be the $k$-stack whose objects over $T$ are
  $T$-flat quasi-coherent $\ms X_T$-twisted sheaves $F$ of finite
  presentation such that for each geometric point $t\in T$, the fiber
  $F_t$ has twisted Mukai vector $(0,\ms O(E),na(\ms X)-1)$ and is $H$-slope-stable.
\end{defn}

In particular, each fiber sheaf $F_t$ above is required to be pure (part of slope-stability), necessarily of dimension $1$.

\begin{defn}\label{sec:modul-interpr-isom-10}
We define two relative stacks.
\begin{enumerate}
\item Let $\ms R^{\bbig}_{\ms X}\to\P^1$ be the stack whose objects over
  $T\to\P^1$ are $T$-flat quasi-coherent $\ms X\times_{\P^1}T$-twisted
  sheaves $F$ of finite presentation such that for each geometric
  point $t\in T$, the pushforward of the fiber $F_t$ along the natural
  closed immersion
  $$\ms X\times_{\P^1}t\to\ms X\times t$$
  is $H$-slope stable with twisted Chern class $(0,\ms O(E),na(\ms X)-1)$.
  
 \item Let $\ms R_{\ms X}$ for the reduced closed substack given by the closure of the preimage of the generic point of $\P^1$.
 \end{enumerate}
\end{defn}

By the usual results on stability (summarized in Section 3.2.1 of \cite{MR2388554}), $\ms M_{\ms X}$ is a
$\G_m$-gerbe over an algebraic space $M_{\ms X}$ and $\ms R_{\ms X}$
is a $\G_m$-gerbe over an algebraic space $R_{\ms X}\to\P^1$.

\begin{remark}
 The reader will note that we define the stability condition in terms of the pushforward of the family to $\ms X$, rather than in the usual classical way, in terms of a relative polarization on $X$ over $\P^1$. This is done in order to avoid dealing with Hilbert polynomials on gerbes -- which are not purely cohomological in nature -- in the case of a singular variety (such as a singular fiber of the pencil). 
\end{remark}

\begin{example}\label{ex:trivial-example}
  When $\ms X\to X$ is the trivial gerbe $X\times\B\m_n$, we can
  compare this to a classical moduli problem. There is an invertible
  $\ms X$-twisted sheaf (corresponding to the natural inclusion
  character $\m_n\inj\G_m$) $\ms L$ such that $\ms L^{\tensor
    n}\cong\ms O_{\ms X}$. Tensoring with $\ms L^\vee$ and pushing
  forward to $X$ defines an isomorphism between $\ms M_{\ms X}$ and
  the stack of coherent pure $1$-dimensional sheaves on $X$ with
  determinant $\ms O(E)$ and second Chern class $-1$. As shown in
  Section 4 of \cite{MR1629929}, this stack is isomorphic to $\ms R^{\bbig}_{\ms X}$, which is isomorphic to the relative moduli
  stack of stable sheaves on the fibers of $f:X\to\P^1$ of rank $1$ and
  degree $1$, and moreover $\ms R_{\ms X}$ is isomorphic to $X$ over $\P^1$. 
  (Note that {\it loc.\ cit.} works over $\C$ and only considers certain components of the moduli space. However, the arguments there do not depend on the base field. If one is willing to believe that the moduli space fibers over $\P^1$ by an elliptic fibration with smooth total space -- following Lemma \ref{sec:modul-interpr-isom-2} below -- an alternative argument to see minimality of the fibration is provided by appealing to Mukai's results on the symplectic structure on the moduli space of sheaves on a K3 surface, Theorem 0.1 of \cite{MR751133}. This then implies that the moduli space is isomorphic to $X$, as desired.) 
\end{example}

\begin{lem}
  \label{sec:modul-interpr-isom-2}
  The stack $\ms M_{\ms X}$ is a $\G_m$-gerbe over a smooth and
  separated scheme $M_{\ms X}$ of dimension $2$.
\end{lem}
\begin{proof}
  Since $\ms M_{\ms X}$ parametrizes stable sheaves, it is a
  $\G_m$-gerbe over its sheafification. Thus, the results will follow
  if we show that for any $F\in\ms M_{\ms X}(k)$ the miniversal
  deformation space is of dimension $2$. The scheme is separated
  because there is a unique stable limit by Langton's theorem 
  (for twisted sheaves, as explained in Lemma 2.3.3.2 of \cite{MR2309155}).

  Recall that there is an obstruction
  theory for $F$ with values in
  $$\ker(\Tr:\Ext^2(F,F)\to\H^2(X,\ms O))$$
  and a deformation theory with values in
  $$\ker(\Tr:\Ext^1(F,F)\to\H^1(X,\ms O)).$$
  By Serre duality, the obstruction theory is dual to the cokernel of
  the natural inclusion map
  $$\Gamma(X,\ms O)\to\Hom(F,F),$$
  which is trivial. The Riemann-Roch theorem shows that $\chi(F,F)=0$,
  and it follows from stability (and Serre duality) that
  $$\dim\Ext^1(F,F)=2.$$
  This shows that $\ms M_{\ms X}$ is smooth, as desired.
\end{proof}

\begin{lem}
  \label{sec:modul-interpr-isom-8}
  Pushforward defines an isomorphism
  $$\phi:\ms R^{\bbig}_{\ms X}\to\ms M_{\ms X}$$
  of $k$-stacks.
\end{lem}
\begin{proof}

  First we define the morphism.
  Fix a $k$-scheme $T$. A point of $\ms R_{\ms X}$ is given by a lift
  $T\to\P^1$ and a $T$-point as in Definition
  \ref{sec:modul-interpr-isom-10}. But the stability condition is
  preserved under the pushforward
  $$\ms X\times_{\P^1}T\to\ms X\times T$$
  by definition of $\ms R_{\ms X}$. Hence, pushing forward along this
  morphism gives an object of $\ms M_{\ms X}$, giving the desired morphism.

  To show that $\phi$ is an isomorphism of stacks, we will show that it is a
  proper monomorphism (hence a closed immersion) that is surjective
  on $k$-points (hence an isomorphism, as $\ms M$ is smooth). We first
  make the following claim.

  \begin{claim}
    Given a morphism $a:T\to\P^1$ and a family $F$ in $\ms R_{\ms
      X}(T)$ with pushforward $\iota_\ast F$ on $\ms X\times T$, we
    can recover the graph of $a$ as the Stein factorization of the morphism
    $$\Supp(\iota_\ast F)\to \P^1\times T,$$
    where $\Supp(\iota_\ast F)$ denotes the scheme-theoretic support
    of $\iota_\ast F$.
  \end{claim}
  \begin{proof}
    Since $X\to\P^1$ is cohomologically flat in dimension $0$, the
    claim follows if the natural map
    $$\ms O_{\ms X\times_{\P^1}T}\to\End(F)$$
    is injective (as it is automatically compatible with base change
    on $T$). By the assumption about the determinant of the fibers of
    $F$, for each geometric point $t\to T$ we know that $\ms O_{\ms
      X_t}\to\End(F_t)$ is injective (as $F_t$, supported on one fiber
    of $X\to\P^1$, must have full support for the determinant on $\ms
    X$ to be correct). The result now follows from Lemma 3.2.3 of
    \cite{MR2579390}.
  \end{proof}

  Suppose $$t_1,t_2:T\to\P^1$$
  are two morphisms and $F_i$ is an object of $\ms R_{\ms X}(t_i)$ for
  $i=1,2$. Write $$\iota_i:\ms X\times_{\P^1,t_i}T\to\ms X\times T$$
  for the two closed immersions.  If
  $\phi(F_1)\cong\phi(F_2)$ then their scheme-theoretic supports
  agree, whence their Stein factorizations agree. By the Claim, the
  two maps $t_1$ and $t_2$ must be equal. But then $F_1$ and $F_2$
  must be isomorphic because their pushforwards are isomorphic.

  Now let us show that $\phi$ is proper. Fix a complete dvr $R$ over
  $k$ with fraction field $K$, a morphism $\Spec K\to\P^1$, and an
  object $F_K\in\ms R_{\ms X}(K)$. Let $\Spec R\to\P^1$ be the unique
  extension ensured by the properness of $\P^1$ and let $$\iota:\ms
  X\tensor_{\P^1}R\to\ms X\tensor_k R$$ be the natural closed
  immersion. We wish to show that the unique stable
  limit $F$ of $\iota_\ast F_K$ has the form $\iota_\ast F$ for an
  $R$-flat family of coherent $\ms X_R$-twisted sheaves (as the
  stability condition then follows by definition).

  Let $\ms I$ be the ideal of the image of $\iota$. By assumption, the
  map of sheaves
  $$\nu:\ms O_{\ms X\tensor_k R}\to\send(F)$$
  kills $\ms I$ in the generic fiber over $R$. Since $F$ is $R$-flat,
  so is $\send(F)$ (as $R$ is a dvr). Thus, the image of $\ms I$ in
  $\send(F)$ is $R$-flat. But this image has trivial generic fiber,
  hence must be trivial. It follows that $\nu$ kills $\ms I$, whence
  $F$ has a natural structure of pushforward along $\iota$, as desired.
\end{proof}

\begin{cor}
  \label{sec:modul-interpr-isom-6}
  The morphism $\ms R_{\ms X}\to\P^1$ is a $\G_m$-gerbe over a smooth surface that is flat over $\P^1$.
\end{cor}
\begin{proof}
  Indeed, the stack $\ms R^{\bbig}_{\ms X}$ is smooth, whence $\ms R_{\ms X}$ is a union of connected components in a smooth stack of dimension $2$. The morphism $\ms R_{\ms
    X}\to\P^1$ is dominant by definition and flat because $\ms R_{\ms X}$ is integral.
\end{proof}

\begin{cor}
  \label{sec:modul-interpr-isom-12}
  For any geometric point $p\to\P^1$ and any object $F$ of $\ms R_{\ms X}(p)$,
  there is a dvr $A$, a diagram
  $$\xymatrix{& p\ar[d]\ar[dl]\\
    \spec A\ar[r] & \P^1}$$
  whose horizontal arrow is dominant, and a family $\ms F\in\ms R_{\ms
    X}(A)$ such that $\ms F_p\cong F$. In particular, any sheaf on a
  singular fiber sits in a flat family with a sheaf on the generic fiber.
\end{cor}
\begin{proof}
  This follows from the flatness of $R_{\ms X}\to\P^1$: one can take a
  general slice through the image of $[F]$ and then take a finite
  normal covering to
  split the restriction of the $\G_m$-gerbe $\ms R_{\ms X}\to R_{\ms
    X}$ to the slice.
\end{proof}

\begin{prop}
  \label{sec:modul-interpr-isom-13}
  Suppose $\ms X$ is a $\m_n$-gerbe that deforms the trivial gerbe. 
  The following hold for the moduli space $R_{\ms X}$ and the $\G_m$-gerbe $\ms R_{\ms X}\to R_{\ms X}$.
  \begin{enumerate}
  \item The morphism
  $$R_{\ms X}\to\P^1$$
  is an \'etale form of the morphism
  $$X\to\P^1.$$ In particular, $R_{\ms X}$ is naturally an elliptic K3 surface.
  \item The association $\ms X\mapsto[R_{\ms X}]$ gives the image of
    the Brauer class of $\ms X$ under the Artin-Tate isomorphism.
  \item The universal sheaf defines a Fourier-Mukai equivalence
    $$\D^{-\tw}(\ms R_{\ms X})\simto\D^{\tw}(\ms X).$$
  \end{enumerate}
\end{prop}

\begin{proof}[Proof of Proposition \ref{sec:modul-interpr-isom-13}]
  Let us first check the second statement. It suffices to verify this
  over the generic point $\eta$ of $\P^1$, so that we may assume $X$ and $R$
  are genus $1$ curves over $k(t)$. The Leray spectral sequence and
  Tsen's theorem show that the edge map gives an isomorphism
  $$\Br(X_\eta)\simto\H^1(\eta,\Jac(X_\eta)),$$
  which we can describe concretely as follows. Over $\widebar{k(t)}$
  the gerbe $\ms X_\eta\to X_\eta$ has trivial Brauer class, hence
  carries an invertible twisted sheaf $\Lambda$ such that
  $\Lambda^{\tensor n}$ has degree $na(\ms X)$, which equals $0$ by our assumption that $\ms X$ deforms the trivial gerbe. 
  Given an element
  $\sigma$ of
  the Galois group of $\widebar{k(t)}$ over $k(t)$, there is an
  invertible sheaf $L_\sigma\in\Pic(X_{\widebar{k(t)}})$ such that
  $\sigma^\ast\Lambda\tensor\Lambda^{\vee}\cong L_{\sigma}|_{\ms
    X_\eta}$. This defines a $1$-cocycle in the sheaf
  $\Pic_{X_\eta/\eta}$, and its cohomology class is the image of a
  unique class in $\H^1(\eta,\Jac(X_\eta))$, as desired.

  On the other hand, tensoring with $\Lambda^{\vee}$ gives an
  isomorphism between the stack of invertible $\ms X_\eta$-twisted
  sheaves of degree $1$ and the stack of invertible sheaves
  on $X_\eta$ of degree $1$. The latter stack is a gerbe over
  $X_\eta$, and the Galois group induces the cocycle given by the
  translation action of $\Jac(X_\eta)$ on $X_\eta$. But this gives the
  edge map in the Leray spectral sequence. 
  This proves the
  second statement.

  Now let us show that $R_{\ms X}$ is a form of $X$ over
  $\P^1$. This turns out to be surprisingly subtle, and uses our assumption that $\ms X$ deforms the trivial gerbe in an essential way. We begin with a lemma.

  \begin{lem}
    \label{sec:modul-interpr-isom-15}
    Given a fiber $D\subset X$ of $f$, there is an invertible $\ms
    X\times_X D$-twisted sheaf $\Lambda$ of rank $1$ such that for
    each smooth curve $C\to D$ the restriction of $\Lambda$ to $C$ has
    degree $0$.
  \end{lem}
  \begin{proof}
    We may replace $D$ with its induced reduced structure, so we will
    assume that $D$ is a reduced curve supported on a fiber of
    $f$. Write $\pi:D\to\spec k$ for the structure morphism.
    Let $T$ be a smooth curve with two points $0$ and $1$ and $\ms
    Y\to X_T$ a $\m_n$-gerbe such that $\ms Y_0\cong\B\m_n$ and $\ms
    Y_1\cong\ms X$ (i.e., a curve connecting $\ms X$ to the trivial gerbe). 
    Let $\ms Z\to D_T$ be the restriction to $D$. The
    gerbe $\ms Z$ gives rise to a morphism of fppf sheaves
    $$T\to \R^2\pi_\ast\m_n.$$
    The Kummer
    sequence shows that there is an isomorphism of sheaves
    $$\Pic_{D/k}/n\Pic_{D/k}\simto\R^2\pi_\ast\m_p.$$ Thus, the gerbe
    $\ms Z$ gives rise to a morphism
    $$h:T\to\Pic_{D/k}/n\Pic_{D/k}$$
    under which $0$ maps to $0$ (by assumption).

    On the other hand, there is a multidegree morphism of $k$-spaces
    $$\Pic_{D/k}\surj\prod_{i=1}^m\Z,$$
    where $m$ is the number of irreducible components of $D$. (This
    map comes from taking the degree of invertible sheaves pulled back
    to normalizations of components, and surjectivity is a basic
    consequence of the ``complete gluing'' techniques of \cite{MR1432058}.) This
    gives rise to a morphism
    $$\deg_n:\Pic_{D/k}/n\Pic_{D/k}\to\prod_{i=1}^m\Z/n\Z$$
    of sheaves.  
      Composing with $h$, it follows from the connectedness of $T$ that $1\in T(k)$ must map into
      the kernel of $\deg_n$.

      By Tsen's theorem, there is an
      invertible $\ms D$-twisted sheaf $\Lambda$, and the above
      calculation shows that $\Lambda^{\tensor n}$ is the pullback of
      an invertible sheaf $L$ on $D$ such that for each irreducible
      component $D_i\subset D$, the pullback of $L$ to the
      normalization of $D_i$ has degree divisible by $n$. Let
      $$\lambda_i\in\Pic(D_i)$$ be an invertible sheaf whose pullback
      to the normalization has degree $-1$. A simple gluing argument
      shows that there is an invertible sheaf $$\lambda\in\Pic(D)$$
      such that
      $$\lambda|_{D_i}\cong\lambda_i$$
      for each $i$. Replacing $\Lambda$ by $\Lambda\tensor\lambda$
      yields an invertible $\ms D$-twisted sheaf whose restriction to
      each $D_i$ has degree $0$, yielding the desired result (as any
      non-constant $C\to D$ factors through a $D_i$).
  \end{proof}

  To show that $R_{\ms X}$ is an \'etale form of $X$, we may
  base-change to the Henselization $U$ of $\P^1$ at a closed point. By
  Lemma \ref{sec:modul-interpr-isom-15}, there is an invertible
  twisted sheaf $\Lambda_u$ on the closed fiber $$X_u\subset X_U$$
  whose restriction to each irreducible component has degree
  $0$. Since the obstruction to deforming such a sheaf lies
  in $$\H^2(X_u,\ms O)=0,$$ we know that $\Lambda_u$ deforms to an
  invertible $\ms X_U$-twisted sheaf $\Lambda$ whose restriction to
  any smooth curve in any fiber of $X_U$ over $U$ has degree $0$.

  Tensoring with $\Lambda$ gives an isomorphism of stacks
  $$\Sh_{\ms X_U/U}(0,\ms O(E),-1)\simto\Sh_{X_U/U}(0,\ms O(E),-1).$$
  We claim that this isomorphism preserves $H$-stability. Since
  stability is determined by Hilbert polynomials, it suffices to prove
  the following.

  \begin{claim}
    \label{sec:modul-interpr-isom-14}
    For any geometric point $u\to U$ and any coherent $\ms X$-twisted
    sheaf $G$, the geometric Hilbert polynomial of $\iota_\ast G$
    equals the Hilbert polynomial of $\Lambda^{\vee}\tensor G$.  In
    particular, $G$ is stable if and only if $\Lambda^{\vee}\tensor G$
    is stable.
  \end{claim}
  \begin{proof}
    Since $G$ is filtered by subquotients supported on the reduced
    structure of a single irreducible component of $D$, it suffices to
    prove the result for such a sheaf. Let $\nu:C\to D$ be the
    normalization of an irreducible component. The sheaves
    $\nu_\ast\nu^\ast G$ and $G$ differ by a sheaf of finite
    length. Thus, it suffices to prove the result for twisted sheaves
    on $C$ and twisted sheaves of finite length. In either case, we
    are reduced to showing the following:
    given a finite morphism $q:S\to X_u$ from a smooth
    $\kappa(u)$-variety, let $\ms S\to S$ be the pullback of $\ms
    X_u\to X_u$. Then for any coherent $\ms S$-twisted sheaf $G$, the
    geometric Hilbert polynomial of $q_\ast G$ equals the Hilbert
    polynomial of $\Lambda^{\vee}\tensor q_\ast G$.

    Using the Riemann-Roch theorem for geometric Hilbert polynomials,
    the classical Riemann-Roch theorem, and the projection formula, we see
    that it is enough to prove that the geometric Hilbert polynomial
    of $G$ (with respect to the pullback of $H$ to $S$) equals the
    usual Hilbert polynomial of $\Lambda_{\ms S}^{\vee}\tensor G$,
    under the assumption that $\Lambda_{\ms S}$ has degree $0$.

    Let $L\in\Pic^0(S)$ be the sheaf whose pullback to $\ms S$
    isomorphic to $\Lambda^{\tensor n}$. Using the isomorphism between
    $K(\ms S)\tensor\Q$ and $K(X)\tensor\Q$, the geometric Hilbert
    polynomial of $G$ is identified with the usual Hilbert polynomial
    of the class
    $$(\Lambda^{\vee}\tensor G)\tensor\frac{1}{n}L^\vee.$$
    But $L\in\Pic^0(S)$, so this Hilbert polynomial is the same as the
    Hilbert polynomial of $\Lambda^{\vee}\tensor G$, as claimed.
  \end{proof}

  We conclude that $R_{\ms X}$ is an \'etale form of the moduli space
  of stable sheaves on $X\to\P^1$ of rank $1$ and degree $1$ on
  fibers (again using the assumption that $\ms X$ deforms the trivial gerbe). 
  This is isomorphic to $X$ itself (see Example \ref{ex:trivial-example}).

  Finally, we need to prove that the universal sheaf defines an
  equivalence of derived categories. It is enough to show that the
  usual adjunction maps are quasi-isomorphisms (see, e.g., Proposition 3.3 of \cite{lieblich}). This shows that it is
  enough to establish the result \'etale-locally on $\P^1$. But then
  it is enough to show the result for an \'stale form of the problem, which
  means that it is enough to show that the ideal sheaf of the diagonal of
  $X\times_{\P^1}X$ gives an equivalence of derived categories. The formula for the Fourier-Mukai transform shows that the ideal sheaf of the diagonal gives the identity map, which is an equivalence, as desired.
\end{proof}

\begin{remark}
As we saw in the proof of Lemma \ref{sec:modul-interpr-isom-15}, the assumption in Proposition \ref{sec:modul-interpr-isom-13} that $\ms X$ deforms the trivial gerbe yields a kind of homogeneity of degrees of restrictions of $\ms X$ to components of fibers. This ensures that the resulting moduli problem can be compared with the classical stable sheaf theory on the underlying family of curves (\'etale-locally on the base). It is not fantastically clear to me at the present moment what happens without this hypothesis.
\end{remark}

\begin{cor}\label{C:family-cor}
Suppose given a $\m_n$-gerbe $\ms X\to X_T$ over a connected base $T$, and a point $t\in T(k)$ such that $\ms X_t$ is trivial. 
The relative moduli stack $$\ms R^{\bbig}_{\ms X}\to T$$ contains an open substack $$\ms R_{\ms X}\to T$$ whose geometric fiber over any $t\to T$ satisfies the conclusion of Proposition \ref{sec:modul-interpr-isom-13}. In particular, a gerbe $\ms X$ on $X\times T$ gives rise to a smooth family of surfaces $R_{\ms X}$ over $T$ with a morphism $R_{\ms X}\to\P^1_T$ realizing each fiber as an \'etale form of $X\to\P^1$.
\end{cor}
\begin{proof}
By Lemma \ref{sec:modul-interpr-isom-2} and Lemma \ref{sec:modul-interpr-isom-8}, the morphism 
$\ms R^{\bbig}_{\ms X}\to T$ is smooth of relative dimension $2$. By smoothness, the functor of connected components of fibers is represented by an \'etale scheme $C$ over $T$. The condition defining $\ms R_{\ms X}$ is an open subset $C'\subset C$. By Proposition \ref{sec:modul-interpr-isom-13}, every geometric fiber of $C'$ is a singleton. It follows that $C'\to T$ is an isomorphism, as desired.
\end{proof}

\begin{cor}
  \label{sec:modul-interpr-isom-9}
  Given a field $L/k$ and a $\m_n$-gerbe $\ms X\to X_L$ deforming the
  trivial gerbe, there is a natural isomorphism of numerical Chow
  groups
  $$\Chow(X)\tensor\Q\simto\Chow(R_{\ms X})\tensor\Q.$$
  In particular, if $\Br(L)=0$ then any class in $\Pic(X_{\widebar
    L})$ is defined over $L$.
\end{cor}
\begin{proof}
  This follows from the cohomological form of Fourier-Mukai
  equivalence combined with the isomorphism in rational Chow theory
  $$\Chow(\ms X)\tensor\Q\simto\Chow(X)\tensor\Q$$ induced by pushforward.

  To see that any class is algebraic, note that $\Pic_{R_{\ms X}/L}$
  must have all of its points defined over separable extensions (as
  the Picard functor of a K3 surface is unramified). Thus, the points
  of $\Pic(R_{\ms X})$ are the Galois invariants in $\Pic(R_{\ms
    X}\tensor L^{\textrm{sep}})$, and these points compute the Picard
  group by the assumption that $\Br(L)=0$.

  By assumption, the rank of $\Chow(R_{\ms X})$ is $24$. It follows
  that the Galois action on $$\Pic(R_{\ms X}\tensor
  L^{\textrm{sep}})\tensor\Q$$ is trivial, whence the action on the
  lattice $\Pic(R_{\ms X}\tensor L^{\textrm{sep}})$ is trivial, as desired.
\end{proof}

\begin{cor}
  \label{C:no-sections}
  If $X$ is a supersingular K3 surface of Artin invariant $10$ then no
  elliptic pencil on $X$ has a multisection of degree prime to $p$.
\end{cor}
\begin{proof}
  First, suppose there is some genus $1$ pencil $\pi:X\to\P^1$ with a
  section. Since $X$ has Artin invariant $10$, any family 
  $$\mc X\to\spec k\[t\]$$
  with supersingular generic fiber and special fiber will yield a restriction
  isomorphism
  $$\Pic(\mc X)\simto\Pic(X).$$
  In particular, as in Lemma 2.3 of \cite{Lieblich:2011ab}, any multisection of $\pi$ will deform in any deformation
  of $\pi$.

  In the situation of Proposition \ref{sec:modul-interpr-isom-13}
  applied to the universal element of $\widehat{\Br}(X)$, we know that
  the geometric generic fiber pencil must be a non-trivial form of
  $\pi$ (over $k\(t\)$). It follows that the generic fiber pencil
  cannot have a multisection of degree prime to $p$. By the previous paragraph, we see that the
  special fiber thus cannot have a section, as claimed.
\end{proof}

\section{Rational curves in moduli spaces}
\label{sec:ratcurves}

Fix a supersingular
elliptic K3 surface $X$ with Artin invariant $10$ and let
$\tau:N\simto\NS(X)$ be a marking by the standard K3 lattice of Artin
invariant $10$. In addition, fix an elliptic pencil $f:X\to\P^1$ and an ample divisor $H\subset X$. We will write $\ms P$ for the period space of $N$-marked
supersingular K3s defined by Ogus in \cite{MR717616}.

The Leray spectral sequence for $\m_p$ with respect to $f$ 
and the vanishing of cohomology of $\A^1$ give a class $$\widetilde\alpha\in\H^2(X\times\A^1,\m_p)$$ that induces a closed and open immersion $$\A^1\inj\R^2\pi_\ast\m_p$$ onto the connected component of the identity. 
Let $\ms X\to X\times\A^1$ be a $\m_p$-gerbe representing $\widetilde\alpha$. 

Every fiber of $\ms X$ over $\A^1$ deforms the trivial gerbe, since $\ms X_0$ parametrizes the trivial class in $\H^2(X,\m_p)$, so we can apply Corollary \ref{C:family-cor}. In particular, we can form the relative moduli space of stable twisted sheaves and use the fiberwise calculations of Section \ref{sec:mod-interp}. Write $Y:=R_{\ms X}\to\A^1$ for the family of moduli spaces; there is a morphism $Y\to\P^1\times\A^1$ such that each fiber over $\A^1$ is an \'etale form of $X\to\P^1$. Moreover, there is a canonical isomorphism $Y_{0} = X$, giving a marking of $Y_0$.

\begin{prop}
  \label{P:form-marking}
  The marking $$\tau:N\simto\NS(Y_0)$$ extends to a marking
  $$N\simto\Pic_{Y/T},$$
  giving a family in $\ms P(T)$.
\end{prop}
\begin{proof}
  By Corollary \ref{sec:modul-interpr-isom-9}, we know that the fibers of $Y$ over (not
  necessarily geometric) points of $\A^1$ have Picard group of rank $22$ (over their fields of definition)
  $22$. On the other hand, for each point $\Spec L\to \A^1$ with $L$ of transcendence degree at most $1$ over $k$, 
  we have (by Tsen's theorem) that
  $\Pic(Y_L)$ is the Galois-invariants in $\Pic(Y_{\widebar
    L})$. Since the invariants have rank $22$, it follows that the
  Galois action on $\Pic(Y_{\widebar L})\tensor\Q$ is trivial,
  whereupon $\Pic(Y_L)=\Pic(Y_{\widebar L})$ by the semisimplicity of finite groups in characteristic $0$.

  Since $Y_0$ has Artin invariant $10$, the (injective)
  specialization map
  $$\Pic(Y_{\widebar k\(t\)})\to\Pic(Y_0)$$ must be an isomorphism. By Popescu's theorem, this descends to some generically \'etale extension of $\A^1$ with non-empty fiber over $0$.  
  Combining this with the previous statement, 
  we see that the specialization map
  $$\Pic(Y_{k(t)})\to\Pic(Y_0)$$ over $\A^1$ is an isomorphism. There is a canonical
  injection $$\Pic(Y_{k(t)})\to\Pic(Y)$$ given by taking closure of divisors, noting that every fiber is smooth, so that closure is a group homomorphism. This gives
  rise to a global marking
  $$N\to\Pic(Y)$$ extending $\tau$, as desired.
\end{proof}

\begin{thm}\label{T:curves in period space}
  Given $(X,\tau)$ and $f:X\to\P^1$, $\sigma:\P^1\to X$ as above, let
  $$c:\G_a\to\ms P$$
  be the morphism induced by $\phi$ and Proposition
  \ref{P:form-marking}, so that $c(0)=(X,\tau)$. The morphism $c$ is
  non-constant. In particular, every point of Artin invariant $10$ in
  $\ms P_N$ lies in a non-trivial rational curve.
\end{thm}
\begin{proof}
  It remains to prove that $c$ is non-constant. But this follows from
  Corollary \ref{sec:modul-interpr-isom-1} 
  and the fact that the geometric Brauer
  class over the generic point of $(\R^2\pi_\ast\m_p)^\circ$ is not
  the pullback of a class over $k$.
\end{proof}

As we will see in the next section, distinct elliptic structures on $X$ gives rise to physically distinct curves in $\ms P$ (i.e., distinct even after reparametrization), showing that this construction yields an infinite collection of rational curves in $\ms P$ through every general point. (It is somewhat surprising that this is true. One naively expects that the orbits of the automorphism group of $X$ acting on the elliptic pencils should parametrize these special rational curves, but that turns out to be too pessimistic.)

\section{Static pencils and deformations} 
\label{sec:static_pencils_and_deformations}
Let $X$ be a supersingular K3 surface of Artin invariant $10$. Since
the Artin invariant is maximal, for any deformation of $X$ over a
Henselian local ring, say
$$\xymatrix{X\ar[r]\ar[d] & \mc X\ar[d]\\
\Spec k\ar[r] & \Spec R,}$$
we have that the restriction map $\Pic(\mc X)\to\Pic(X)$ is an isomorphism.

\begin{lem}
  \label{L:pencil-deforms}
  Suppose $f:X\to\P^1$ is a pencil of genus $1$ curves on $X$. Given a
  deformation $\mc X/R$ as above, there is a deformation of $f$ to a
  relative pencil $F:\mc X\to\P^1_R$.
\end{lem}
\begin{proof}
  Let $E$ be a smooth fiber of $f$. First, since $X$ has Artin
  invariant $10$, the class of $E$ lifts over $R$ to some invertible
  sheaf $\ms L\in\Pic(\mc X)$. It follows from basic deformation
  theory and the vanishing of $\H^1(X,\ms O(E))$ that in fact the
  global sections of $\ms L|_X$ lift to sections of $\ms L$ over $\mc
  X$. This lifts the pencil.
\end{proof}

\begin{defn}
  Suppose $\mc X\to R$ is a deformation of $X$ over a Henselian local
  $k$-algebra $k\to R$.  A pencil $f:X\to\P^1$ is \emph{static with
    respect to the deformation $\mc X_{/R}$} if there is a lift of the
  pencil $F:\mc X\to \P^1_R$ such that the pencils $f\tensor R$ and
  $F$ are isomorphic \'etale-locally on $\P^1_R$.
\end{defn}

\begin{defn}
  Two pencils $f:X\to\P^1$ and $g:X\to\P^1$ are \emph{distinct} if
  there is no commutative diagram
  $$\xymatrix{& X\ar[dr]^f\ar[dl]_g &\\
    \P^1\ar[rr] & & \P^1}$$ of isomorphisms. Equivalently, the fibers
  of $f$ and $g$ are not linearly equivalent.
\end{defn}

\begin{defn}
  Two pencils $f:X\to\P^1$ and $g:X\to\P^1$ are \emph{transverse} if
  there is a smooth fiber $E:=f^{-1}(x)$ and a smooth fiber
  $D:=g^{-1}(x')$ such that $E\cap D$ is reduced and $\ms
  O(E)\not\cong\ms O(D)$. In
  other words, general fibers of $f$ and $g$ are not linearly
  equivalent and they intersect
  transversely. Equivalently, the restriction of the morphism $f$ to a
  general fiber of $g$ is generically \'etale (or vice versa).
\end{defn}

\begin{prop}
  \label{P:static-implies-constant}
  Suppose $R$ is a Henselian local augmented
  $k$-algebra. If $\mc X\to\spec R$ is a deformation
  of $X$ over which two transverse pencils remain static then $\mc X$ is
  isomorphic to the constant deformation $X_R$.
\end{prop}
\begin{proof}
  Write $$f_1,f_2:X\to\P^1$$ for the pencils with static
  lifts $$F_1,F_2:\mc X\to\P^1_R.$$ Write $P=\P^1_R\times\P^1_R$ and
  $\widehat P$ for the formal scheme given by completing $P$ along the
  augmentation ideal of $R$.

  Since $f_1$ and $f_2$ are distinct, we have two finite maps
  $$\Phi=(F_1,F_2):\mc X\to P$$ and
  $$\phi=((f_1)_R,(f_2)_R):X_R\to P.$$ 
  Moreover, since the pencils are transverse, we know that $\Phi$ and
  $\phi$ are generically \'etale. Write $\widehat{\mc X}$ and
  $\widehat{X_R}$ for the formal schemes given by completing each
  along the augmentation ideal of $R$, so that there are finite
  generically \'etale morphisms of formal schemes
  $$\widehat\Phi:\widehat{\mc X}\to\widehat P$$ and 
  $$\widehat\phi:\widehat{X_R}\to\widehat P.$$

  Under the assumption that both $f_1$ and $f_2$ are static pencils,
  we have that both $\Phi$ and $\phi$ are isomorphic \'etale-locally
  on $P$. The \'etale sheaf $\Isom_P(\Phi,\phi)$ is thus a torsor
  under the sheaf $\Aut(\phi)$, and similarly for $\Isom_{\widehat
    P}(\widehat\Phi,\widehat\phi)$. By the infinitesimal rigidity of the \'etale
  site, we can identify the small \'etale site of $\widehat P$
  with the small \'etale site of $P\tensor_R k$. Via this
  identification, there is a reduction map
  $$\Aut_{\widehat P}(\phi)\to\Aut_{P_k}(\phi_k).$$
  Since $X$ is integral and $F$ is generically unramified, we know
  that this reduction map is an isomorphism (i.e., there are no
  infinitesimal automorphisms for a generically \'etale finite
  morphism with integral domain). There is thus an induced isomorphism
  $$\H^1(\widehat P_{\textrm{\'et}},\Aut_{\widehat P}(\widehat\phi))\simto
  \H^1((P_k)_{\textrm{\'et}},\Aut_{P_k}(\phi_k)).$$

  Since $\Phi$ and $\phi$ are isomorphic over $k$ (being deformations
  of the same pair of pencils), it follows that the class of the
  torsor $\Isom_{\widehat P}(\widehat{\Phi},\widehat\phi)$ is
  trivial. On the other hand, the Grothendieck existence theorem shows
  that the natural completion map
  $$\H^1(P_{\textrm{\'et}},\Aut_P(\Phi,\phi))\to\H^1((\widehat P)_{\textrm{\'et}},\Aut_{\widehat P}(\widehat\Phi,\widehat\phi))$$
  is injective (in fact, an isomorphism). It follows that
  $\Isom_P(\Phi,\phi)$ is a trivial torsor, which shows that $F$ is
  isomorphic to the trivial deformation, and thus that $\mc X$ itself is
  isomorphic to the trivial deformation of $X$.
\end{proof}

\begin{defn}
  Call two pencils $f,g:X\to\P^1$ \emph{inequivalent} if the fibers of
  $f$ and $g$ over $0$ are not linearly equivalent.
\end{defn}

\begin{prop}
  \label{sec:stat-penc-deform}
  Suppose $R$ is a normal Henselian local augmented $k$-algebra. If
  $\mc X\to\Spec R$ is a deformation of $X$ over which two
  inequivalent pencils remain static then $\mc X$ is isomorphic to the
  constant deformation $X_R$.
\end{prop}
\begin{proof}
  We start with the same morphisms
$$\Phi:\mc X\to P$$ and $$\phi:X_R\to P$$ as in the proof of
Proposition \ref{P:static-implies-constant}, and, as above, we know
that they are isomorphic \'etale-locally on $P$. What we do not know
is that $\Phi$ and $\phi$ are generically \'etale. We can avoid this
since we are working with a normal domain $R$ as the base ring.

First, write $X\to P_k$ as a composition
$$X\to\widebar X\to P_k$$
where $\widebar X\to P_k$ is separable and $X\to\widebar X$ is purely
inseparable. This is canonical (taking $\widebar X$ to be the
normalization of $P$ inside the separable closure of $k(P)$ inside
$k(X)$), and the factorization $$X_R\to\widebar X_R\to P$$ is
identified with the normalization of $P$ inside the separable closure
of $K(P)$ in $K(X_R)$. 

Since $\mc X$ is \'etale-locally isomorphic to $X_R$, it follows that
for the analogous factorization $$\mc X\to{\widebar X}\to P,$$
we know that $\widebar{\mc X}$ is an \'etale form of $\widebar
X_R$. By the arguments in the proof of Proposition
\ref{P:static-implies-constant}, we conclude that $\widebar{\mc
  X}\cong\widebar X_R$. Choose an identification between the two. It
remains to show that the two deformations of $X\to\widebar X$ are
themselves isomorphic, knowing that they are simultaneously purely
inseparable and \'etale-locally isomorphic.

Passing to the generic point of $\widebar X_R$, we have two purely
inseparable field extensions $M_1/L$ and $M_2/L$ such that
$$M_1\tensor_LL^{\textrm{sep}}\cong M_2\tensor_LL^{\textrm{sep}}.$$
Since every scheme in sight is normal, it suffices to show that
$M_1\cong M_2$ (as $L$-algebras). But the automorphism sheaf $\Aut_L(M_1)$
on the small \'etale site of $\Spec L$ is the singleton sheaf (since
$M_1$ and $L^{\textrm{sep}}$ are linearly disjoint, and a purely
inseparable field extension has trivial automorphism group). Thus,
\'etale forms are all trivial, as desired.
\end{proof}

\begin{remark}
  The reader will note the curious fact that the proof of Proposition
  \ref{sec:stat-penc-deform} uses the normality of $R$ in an essential
  way. In particular, we gain no insight into the infinitesimal
  properties of pairs of static pencils that are not transverse. As
  Maulik as pointed out to us, if one works over a finite base field,
  one can deduce from Propositions \ref{P:static-implies-constant} and
  \ref{sec:stat-penc-deform} that given an infinite list of pencils on
  $X$, applying the construction of Theorem \ref{T:curves in period
    space} below yields an infinite list of curves such that for any given finite
  order $n$, infinitely many of these curves must agree up to order $n$ (as there are only finitely many jets of a given order
  on the Ogus space over a finite field). In particular, we cannot
  have infinitely many pairwise transverse pencils on a
  supersingular K3 surface of Artin invariant $10$.
\end{remark}


\begin{prop}
  \label{sec:stat-penc-deform-2}
  Let $X$ be a K3 surface and $\pi:X\to\P^1$ an elliptic
  pencil. The locus in $\Def_{X}$ parametrizing deformations over which
  $\pi$ remains static is $1$-dimensional.
\end{prop}
\begin{proof}
  Let $\Def_{X/\P^1}$ be the functor whose objects over an augmented
  Artinian $k$-algebra $A$ are Cartesian diagrams
$$\xymatrix{X\ar[r]\ar[d] & \mc X\ar[d]\\
  \P^1_k\ar[r] & \P^1_A}$$ in which the vertical arrows are flat and
proper. (In other words, $\Def_{X/\P^1}$ parametrizes relative
deformations of the pencil.)  Let $\Def^s_{X/\P^1}$ denote the
subfunctor parametrizing families that are isomorphic to the
constant family
$$\xymatrix{X\ar[r]\ar[d] & X\times\P^1_A\ar[d]\\
  \P^1\ar[r] & \P^1_A}$$
locally on $\P^1_A$.

\begin{lem}
  \label{sec:stat-penc-deform-4}
The functor $\Def_{X/\P^1}^s$ is prorepresentable. 
\end{lem}
\begin{proof}
  We will temporarily write $F$ for the functor $\Def^s_{X/\P^1}$. To show that $F$ is prorepresentable, 
we will use Schlessinger's criterion. Given morphisms $A\to C$ and $B\to C$ in $\operatorname{Art}_k$, 
there is a natural diagram
\begin{equation}
  \label{eq:sch}
 F(A\times_C B)\to F(A)\times_{F(C)}F(B)
\end{equation}

We need to check the following.
\begin{enumerate}
\item \eqref{eq:sch} is a surjection when $B\to C$ is small
\item \eqref{eq:sch} is a bijection when $C=k$ and $B=k[\eps]$
\item $F(k[\eps])$ is a finite-dimensional vector space (with its natural structure)
\item if $A\to C$ is small then $F(A\times_C A)\to F(A)\times_{F(C)} F(A)$ is a bijection
\end{enumerate}

Since we already know that these conditions hold for the moduli of diagrams $X\to\P^1$, 
the key is showing that they respect the \'etale-local triviality condition. In other words, we 
need to show that given a family
$$\mc X\to\P^1_{A\times_C B}$$
such that the restricted families
$$\mc X_A\to\P^1_A$$
and $$\mc X_B\to\P^1_B$$
are \'etale-locally isomorphic to the trivial family, then the same holds for the original family.
But we know that the morphism of $\P^1_{A\times_C B}$-schemes 
$$\Isom_{\P^1_{A\times_C B}}(\mc X,X_{A\times_C B})\to\Isom_{\P^1_A}(\mc X_A,X_A)
\times_{\Isom_{\P^1_C}(\mc X_C,X_C)}\Isom_{\P^1_B}(\mc X_B,X_B)$$
is an isomorphism under all of the listed conditions because the stack of elliptic surfaces is algebraic.
The results now follow from the topological invariance of the \'etale site.
\end{proof}

\begin{lem}\label{sec:stat-penc-deform-5}
  The formal scheme $\Def^s_{X/\P^1}$ is formally smooth and  
$1$-dimensional.
\end{lem}
\begin{proof}
The infinitesimal automorphism sheaf $\ms A$ of $X\to\P^1$ is precisely the normal sheaf
of the $0$-section $\P^1\to J$ of the Jacobian fibration of $X$. 
Since $J$ is also a K3 surface, 
the normal sheaf is $\ms O(-2)$. In particular,  $\H^1(\P^1, \ms A)$ is $1$-dimensional and $\H^2(\P^1,\ms A)=0$. Given a square-zero extension $A\to A_{0}$ with ideal sheaf $I$ and a point $\Def^s_{X/\P^1}(A_0)$, the lifts to $A$ are obstructed by elements of $\H^2(\P^1_{A_0}, \ms A\tensor I)$ and form a pseudo-torsor under $\H^1(\P^1_{A_0}, \ms A\tensor I)$. It follows that deformations are unobstructed (so $\Def^s_{X/\P^1}$ is formally smooth) and the tangent space is $1$-dimensional, as desired.
\end{proof}

\begin{lem}\label{sec:stat-penc-deform-3}
  The forgetful morphism
$$\Def_{X/\P^1}^s\to\Def_X$$
is a closed immersion of formal $k$-schemes
\end{lem}
\begin{proof}
  It is enough to show that the tangent map
$$T\Def^s_{X/\P^1}\to T\Def_X$$
is injective. Suppose $$X_\eps\to\P^1_{k[\eps]}$$ is a tangent vector
that maps to $0$. This means that the underlying surface $X_{\eps}$ is
isomorphic to $X_{k[\eps]}$ in a way compatible with the
identifications over $k$. Choosing such an isomorphism yields two
morphisms
$$f,g:X_{k[\eps]}\to\P^1_{k[\eps]}$$
with the property that for each $k$-point $p\in\P^1$, the restrictions of $f$ and $g$ to 
$p\tensor_k k[\eps]$ are constant. Consider the Stein factorization of the induced morphism
$$(f,g):X_{k[\eps]}\to\P^1_{k[\eps]}\times_{\spec k[\eps]}\P^1_{k[\eps]}.$$
Since every scheme in sight is $\eps$-flat, the Stein factorization is a finite $\eps$-flat 
morphism
$$S\to\P^1\times\P^1$$
over $k[\eps]$. Moreover, $S\tensor_{k[\eps]} k$ is isomorphic to the
  diagonal by the definition of the moduli problems. Thus, $S$ is an
  infinitesimal deformation of the diagonal
  $\Delta\subset\P^1_k\times\P^1_k$. We know that $\Delta^2=2$, so the
  space of infinitesimal deformations is a tosor under $\H^0(\P^1,\ms
  O(2))$. In fact, this is just the tangent space to the automorphism
  group scheme $\PGL_2$ of $\P^1$. Since each $k$-point must be fixed
  (by the static assumption), there are no non-trivial infinitesimal
  automorphisms, and we see that $f=g$, as desired.
\end{proof}

This completes the proof of Proposition \ref{sec:stat-penc-deform-2}.

\end{proof}


\bibliography{artinconjbib}{}

\begin{thebibliography}{10}

\bibitem{MR2786662}
Dan Abramovich, Martin Olsson, and Angelo Vistoli.
\newblock Twisted stable maps to tame {A}rtin stacks.
\newblock {\em J. Algebraic Geom.}, 20(3):399--477, 2011.

\bibitem{MR0371899}
M.~Artin.
\newblock Supersingular {$K3$} surfaces.
\newblock {\em Ann. Sci. \'Ecole Norm. Sup. (4)}, 7:543--567 (1975), 1974.

\bibitem{MR1629929}
Tom Bridgeland.
\newblock Fourier-{M}ukai transforms for elliptic surfaces.
\newblock {\em J. Reine Angew. Math.}, 498:115--133, 1998.

\bibitem{dejong-gabber}
A.~J. de~Jong.
\newblock A result of {G}abber.
\newblock 2003.

\bibitem{MR0244271}
Alexander Grothendieck.
\newblock Le groupe de {B}rauer. {III}. {E}xemples et compl\'ements.
\newblock In {\em Dix {E}xpos\'es sur la {C}ohomologie des {S}ch\'emas}, pages
  88--188. North-Holland, Amsterdam; Masson, Paris, 1968.

\bibitem{huybrechts}
Daniel Huybrechts.
\newblock Lectures on k3 surfaces.
\newblock http://www.math.uni-bonn.de/people/huybrech/K3Global.pdf, 2013.

\bibitem{MR2309155}
Max Lieblich.
\newblock Moduli of twisted sheaves.
\newblock {\em Duke Math. J.}, 138(1):23--118, 2007.

\bibitem{MR2388554}
Max Lieblich.
\newblock Twisted sheaves and the period-index problem.
\newblock {\em Compos. Math.}, 144(1):1--31, 2008.

\bibitem{MR2579390}
Max Lieblich.
\newblock Compactified moduli of projective bundles.
\newblock {\em Algebra Number Theory}, 3(6):653--695, 2009.

\bibitem{MR3114930}
Max Lieblich.
\newblock On the ubiquity of twisted sheaves.
\newblock In {\em Birational geometry, rational curves, and arithmetic}, pages
  205--227. Springer, New York, 2013.

\bibitem{1403.3073}
Max Lieblich.
\newblock On the unirationality of supersingular {K3} surfaces.
\newblock August 2014, 1403.3073.

\bibitem{Lieblich:2011ab}
Max Lieblich and Davesh Maulik.
\newblock A note on the cone conjecture for k3 surfaces in positive
  characteristic.
\newblock 02 2011, 1102.3377.

\bibitem{lieblich}
Max Lieblich and Martin Olsson.
\newblock Fourier-{M}ukai partners of {K3} surfaces in positive characteristic,
  http://arxiv.org/abs/1112.5114.

\bibitem{1304.5623}
Christian Liedtke.
\newblock Supersingular {K3} surfaces are unirational.
\newblock {\em Invent.\ Math.}, 200(3):979--1014, 2014.
\newblock arxiv:1304.5623v4.

\bibitem{MR1432058}
Laurent Moret-Bailly.
\newblock Un probl\`eme de descente.
\newblock {\em Bull. Soc. Math. France}, 124(4):559--585, 1996.

\bibitem{MR751133}
Shigeru Mukai.
\newblock Symplectic structure of the moduli space of sheaves on an abelian or
  {$K3$} surface.
\newblock {\em Invent. Math.}, 77(1):101--116, 1984.

\bibitem{MR717616}
Arthur Ogus.
\newblock A crystalline {T}orelli theorem for supersingular {$K3$} surfaces.
\newblock In {\em Arithmetic and geometry, {V}ol. {II}}, volume~36 of {\em
  Progr. Math.}, pages 361--394. Birkh\"auser Boston, Boston, MA, 1983.

\bibitem{MR2312554}
Martin Olsson.
\newblock Sheaves on {A}rtin stacks.
\newblock {\em J. Reine Angew. Math.}, 603:55--112, 2007.

\bibitem{MR1785328}
Andreas Schweizer.
\newblock Extremal elliptic surfaces in characteristic 2 and 3.
\newblock {\em Manuscripta Math.}, 102(4):505--521, 2000.

\bibitem{MR1312368}
Joseph~H. Silverman.
\newblock {\em Advanced topics in the arithmetic of elliptic curves}, volume
  151 of {\em Graduate Texts in Mathematics}.
\newblock Springer-Verlag, New York, 1994.

\bibitem{MR1610977}
John Tate.
\newblock On the conjectures of {B}irch and {S}winnerton-{D}yer and a geometric
  analog.
\newblock In {\em S\'eminaire {B}ourbaki, {V}ol.\ 9}, pages Exp.\ No.\ 306,
  415--440. Soc. Math. France, Paris, 1995.

\end{thebibliography}
\bibliographystyle{hplain}
\end{document}